\documentclass[a4paper,12pt,openbib,reqno]{amsart}
\usepackage{amsmath,amsfonts,amssymb}
\usepackage{verbatim}
\usepackage{enumerate}
\usepackage{url}
\usepackage[colorlinks]{hyperref}
\usepackage[latin1]{inputenc}
\usepackage[english]{babel}
\usepackage{tikz}
\usepackage[margin=2.5cm]{geometry}
\usepackage{bbm}
\usepackage{multirow}
\usepackage{mathrsfs}
\usepackage{caption}
\usepackage{float}
\restylefloat{table}

\def\Z{\mathbb Z}

\def\Ca{\mathcal C}

\def\ha{\mathcal H}

\def\K{\mathbb K}
\def \F{\mathbb F}

\def\Fq{\mathbb{F}_q}
\def\Fp{\mathbb{F}_p}
\def\Fqn{\mathbb{F}_{q^n}}

\DeclareMathOperator{\aut}{Aut}
\DeclareMathOperator{\Ima}{Im}
\DeclareMathOperator{\tr}{Tr}

\def\Lx{\mathscr{L}_{n,q}[x]}
\def\Lxp{\mathscr{L}_{n,p}[x]}
\def\gab{\mathcal{G}_{k,s}}
\def\hl{\mathcal{H}_{k,s}(L_1,L_2)}
\def\hxl{\mathcal{H}_{k,s}(x,L(x))}

\def\hm{\mathcal{H}_{k,r}(M_1,M_2)}
\def\hms{\mathcal{H}_{k,s}(M_1,M_2)}

\def\lcm{\mathop{\rm lcm}}

\theoremstyle{plain}
\newtheorem{theorem}{Theorem}[section]
\newtheorem{lemma}[theorem]{Lemma}

\newtheorem{corollary}[theorem]{Corollary}
\newtheorem{proposition}[theorem]{Proposition}

\newtheorem{example}[theorem]{Example}

\def\qed{\hfill\hbox{$\square$}}

\theoremstyle{definition}

\author{José Alves Oliveira}
\address{
		Departamento de Matem\'{a}tica\\
		Universidade Federal de Minas Gerais (UFMG)\\
		Belo Horizonte, MG\\
		31270-901\\
		Brazil\\}
\email{jose-alvesoliveira@hotmail.com}

\title{A note on rank-metric codes}

\date{\today
}

\subjclass[2000]{ }


\begin{document}
\baselineskip=1.6\baselineskip
\begin{abstract}
Let $\Fq$ denote the finite field with $q=p^\lambda$ elements. Maximum Rank-metric codes (MRD for short) are subsets of $M_{m\times n}(\Fq)$ whose number of elements attains the Singleton-like bound. The first MRD codes known was found by Delsarte (1978) and Gabidulin (1985). Sheekey (2016) presented a new class of MRD codes over $\Fq$ called twisted Gabidulin codes and also proposed a generalization of the twisted Gabidulin codes to the codes $\hl$. The equivalence and duality of twisted Gabidulin codes was discussed by Lunardoni, Trombetti, and Zhou (2018). A new class of MRD codes in $M_{2n\times 2n}(\Fq)$ was found by Trombetti-Zhou (2018). In this work, we characterize the equivalence of the class of codes proposed by Sheekey, generalizing the results known for twisted Gabidulin codes and Trombetti-Zhou codes. In the second part of the paper, we restrict ourselves to the case $L_1(x)=x$, where we present its right nucleus, middle nucleus, Delsarte dual and adjoint codes. In the last section, we present the automorphism group of $\hxl$ and compute its cardinality. In particular, we obtain the number of elements in the automorphism group of the twisted Gabidulin codes.
\end{abstract}
	
\maketitle
	
\section{Introduction} \label{sec10}
For a field $\K$, let $M_{m\times n}(\K)$ be the set of $m\times n$ matrices over $\K$. A \textit{rank metric code $\Ca$} is a subset of $M_{m\times n}(\K)$ equipped with the distance function $d(A,B)=rank(A-B)$. The rank metric codes were introduced by Delsarte in \cite{delsarte1978bilinear}. The \textit{minimun distance} of $\Ca$ is given by $d(\Ca)=\min_{A,B\in\Ca,A\neq B}\{d(A,B)\}.$ The particular case where $\K$ is a finite field have been studied over the last few decades, since there exist many interesting properties involving these codes. Let $\Fq$ denotes a finite fields with $q=p^\lambda$ elements, where $p$ is a prime number and $\lambda$ is a positive integer. A code $\Ca\subset M_{m\times n}(\Fq)$ that attains the Singleton-like bound 
$$|\Ca|\leq q^{\max\{m,n\}(\min\{m,n\}-d(\Ca)+1)}$$
 is called \textit{maximum rank distance code} (MRD code for short). The first MRD codes over a finite field over $\Fq$ have been constructed by Delsarte \cite{delsarte1978bilinear} and Gabidulin \cite{gabidulin1985theory}. Currently, these codes are often called \textit{generalized Gabidulin codes}.
A matrix $A\in M_{n,n}(\Fq)$ is uniquely represented by a polynomial in 
$$\mathcal{L}_{n,q}[x]:=\left\{\sum a_i x^{q^i}:a_i\in\Fqn\right\},$$
the set of linearized polynomials over $\Fqn$. We note that $\mathcal{L}_{n,q}[x]$ equipped with composition is a group. In this paper we will use the language of linearized polynomials over $\Fq$ (see definition in Section \ref{sec12}) in order to prove our main results. Throughout the paper, we also use the bijection between $A\in M_{n,n}(\Fq)$ and $\mathscr{L}_{n,q}[x]:=\mathcal{L}_{n,q}[x]/(x^{q^n}-x)$. Using the language of linearized polynomials the \textit{generalized Gabidulin code} $\gab$ can be seen as the set
$$\{a_0x+a_1x^{q^s}+\cdots+a_{k-1}x^{q^{s(k-1)}}:a_i\in\Fqn\},$$
where $s$ is relatively prime to $n$. The number of elements in $\gab$ is $q^{nk}$ and each polynomial in it has at most $q^{k-1}$ roots, which means its minimum distance is $d=n-k+1$ and then $\gab$ is an MRD code. 

In the last few years, many authors have presented important contributions to the general theory of MRD codes (e.g see \cite{delsarte1978bilinear,gabidulin1985theory,ravagnani2016rank,gadouleau2006genp1,lunardon2017mrd,morrison2014equivalence}). Sheekey~\cite{sheekey2015new} proposed the study of the generalized twisted Gabidulin codes 
$$\ha_{k,s}(L_1(x),L_2(x))=\{L_1(a_0)x+a_1x^{q^s}+\cdots+L_2(a_0)x^{q^{sk}}:a_i\in\Fqn\}$$
where $L_1(x)$ and $L_2(x)$ are linearized polynomials over $\Fqn$. Let $N_{q^n,q}(a)=a^{\frac{q^n-1}{q-1}}$ be the norm function from $\Fqn$ to $\Fq$ and let $\Fq^*:=\Fq\backslash\{0\}$. By Lemma \ref{item1}, $\ha_{k,s}(L_1(x),L_2(x))$ is an MRD code if $N_{q^n,q}(L_1(a))\neq (-1)^{nk}N_{q^n,q}(L_2(a))$ for all $a\in\Fqn^*$. These codes have been used for many authors in order to present new classes of MDR codes. To date the following MRD codes are known (see survey in \cite{sheekey201913}):
\begin{table}[H]
\begin{center}
	\begin{tabular}{ |c|c|c|c|c| } 
		\hline
		Name & $L_1(x)$ & $L_2(x)$ & Conditions  & Reference\\ \hline
		TG & $x$ & $\eta x^{q^h}$ & $N_{q^n,q}(\eta)\neq (-1)^{nk}$ and $s=1$ & \cite{sheekey2015new}\\ \hline
		GTG & $x$ & $\eta x^{q^h}$ & $N_{q^n,q}(\eta)\neq (-1)^{nk}$  & \cite{lunardon2018generalized}\\ \hline
		AGTG & $x$ & $\eta x^{p^h}$ & $N_{q^n,p}(\eta)\neq (-1)^{nk}$  & \cite{otal2016additive}\\ \hline
		\multirow{3}{*}{TZ} & \multirow{3}{*}{$x+x^{q^{n/2}}$ }& \multirow{3}{*}{$\dfrac{\eta(x-x^{q^{n/2}})}{\theta}$} & $n$ even,  & \multirow{3}{*}{\cite{trombetti2018new}}\\ 
		& &  & $N(\eta)$ is not a quadratic residue in $\Fq$ & \\
		& &  & and $\theta\in\Fqn\backslash \F_{q^{n/2}}$ and $\theta^2\in\F_{q^{n/2}}$  &\\ \hline
	\end{tabular}
\caption{TG=Twisted Gabidulin, GTG=Generalised Twisted Gabidulin,\newline AGTG=Additive Generalised Twisted Gabidulin, TZ = Trombetti-Zhou}\label{tab1}
\end{center}
\end{table}
We say that two codes $\Ca$ and $\Ca'$ are equivalent if there exist $\psi\in\mathscr{L}_{n,q}[x]$, two bijective linearized polynomials $\phi_1,\phi_2\in\mathscr{L}_{n,q}[x]$ and $\rho\in\aut(\Fq)$ such that 
$$\Ca'=\{\phi_1\circ f^\rho\circ\phi_2+\psi:f\in\Ca\}$$
where $f^\rho=\sum a_i^\rho x^{q^i}$ for $f=\sum a_i x^{q^i}$. When $\Ca$ and $\Ca'$ are both additive, it is not
difficult to show that we can choose $\psi=0$. It is always a difficult task to show that a new family of MRD codes is inequivalent to another family already known. Lunardoni, Trombetti, and Zhou \cite{lunardon2018generalized} characterized the equivalence of generalized twisted Gabidulin codes and presented its Delsarte dual and adjoint codes. In Section \ref{sec14} we characterize completely the equivalence between codes of the form $\ha_{k,s}(L_1(x),L_2(x)),$ generalizing the results obtained in \cite{lunardon2018generalized,trombetti2018new} for the codes in the Table \ref{tab1}.

In Section \ref{sec15} we restrict ourselves to the codes of the form $\mathcal{H}_k(x,L(x))$, where we completely describe their nucleus, Delsarte dual codes and adjoint codes (see definitions in Section \ref{sec12}). The most codes in Table \ref{tab1} (namely TG, GTG, AGTG) are covered for our results.

Lastly, in the Section \ref{sec16} we characterize the automorphism group of $\mathcal{H}_k(x,L(x))$ and compute its number of elements. In particular, the automorphism groups of TG, GTG, AGTG codes are obtained.

\section{ Preliminaries}\label{sec12}
	
In this section, we introduce concepts that will be useful throughout the paper. Let define $\ha_{k}(L_1,L_2)=\ha_{k,1}(L_1,L_2)$. For an integer $m\geq 0$ and $a_0,\ldots,a_m\in\Fqn$, a polynomial of the form
$$a_0x+a_1 x^q+\cdots a_{m} x^{q^{m}}$$
is called a linearized polynomial over $\Fqn$. As $a^{q^n}=a$ for all $a\in\Fqn$, we can consider only polynomials with $q$-degree smaller than $n$. The following result is an immediate consequence of the Lagrange interpolation formula for polynomials over $\Fqn$.

\begin{lemma}\label{item5}
	Let $L(x)\in\Lx$. If $L(a)=0$ for all $a\in\Fqn$, then $L(x)\equiv 0$.
\end{lemma}

For $s$ an integer such that $\gcd(s,n)=1$ and an positive integer $k$, the polynomials of the form 
$$a_0x+a_1 x^{q^s}+\cdots a_{k} x^{q^{sk}}$$
are called $q^s$-linearized polynomial. For more properties of linearized polynomials over finite fields, see \cite[Chapter $3.4$]{Lidl}.  Sheekey \cite{sheekey2015new} has been using properties of $q^s$-linearized polynomials in order to find new classes of MRD codes, as we can see below. 
\begin{lemma}\cite[Remark $8$]{sheekey2015new}\label{item1}
	Let $s$ and $m$ be two relatively prime positive integers. Suppose that 
	$$f(x)=\sum_{i=0}^{k} a_i x^{q^{i}}$$
	is a $q^s$-linearized polynomial over $\Fqn$. If $f$ has rank $n-k$, then $N_{q^n,q}(a_0)=(-1)^{kn}N_{q^n,q}(a_0)$
\end{lemma}

As an immediate consequence, Sheekey has shown the following result.

\begin{proposition}\label{item2}
	Let $L_1$ and $L_2$ be  linearized polynomials over $\Fqn$ and $k\leq n-1$. If $N_{q^n,q}(L_1(a))\neq(-1)^{kn}N_{q^n,q}(L_2(a))$ for all $a\in\Fqn^*$, then $\ha_{k}(L_1,L_2)$ is an MRD code.
\end{proposition}

The most MRD codes that are not equivalent to \textit{generalized Gabidulin codes} was found using the Proposition \ref{item2}. There are several invariants for a rank metric codes which we can use to decide when two codes are equivalent to another already known. Next we present some of this invariants. The middle and right nucleus of semifields were introduced in \cite{lunardon2017kernels}. We define the middle nucleus $\mathcal{N}_m$ and right nucleus $\mathcal{N}_r$ for a rank metric code $\Ca$ as
$$\mathcal{N}_m(\Ca)=\{g(x)\in\mathscr{L}_{n,q}[x]: f\circ g\in\Ca\text{ for all }f\in\Ca\}$$
and
$$\mathcal{N}_r(\Ca)=\{g(x)\in\mathscr{L}_{n,q}[x]: g\circ f\in\Ca\text{ for all }f\in\Ca\}.$$

The nuclei of TG and GTG codes can be found in \cite{lunardon2018generalized}. In \cite{trombetti2018new}, the author present the nuclei of the TZ codes. Let $\tr_{q^n/q}(x)=x+x^q+\cdots+x^{q^{n-1}}$ be the trace function of $\Fqn$ over $\Fq$. The \textit{adjoint} of a $q^s$-linearized polynomial $f=\sum_{i=0}^{n-1} a_i x^{q^{si}}$ is given by
$$\hat{f}=\sum_{i=0}^{n-1}a_i^{q^{s(n-i)}} x^{q^{s(n-i)}}$$
and the \textit{adjoint code} of a rank metric code $\Ca$ is $\mathcal{\widehat{C}}=\{\hat{f}:f\in\Ca\}.$ For the codes which we are interested in this paper, it is easy to verify that the following result holds.

\begin{proposition}\label{item15}
	The adjoint code of $\hl$ is equivalent to $\mathcal{H}_{n-k,s}(L_2^{q^{n-sk}},L_1)$.
\end{proposition}

Another useful invariant is the dual of a code. The \textit{Delsarte dual code} of a code $\Ca$ is given by
$$\Ca^\perp=\{g(x)\in\mathscr{L}_{n,q}[x]:b(f,g)=0\text{ for all }f\in\Ca\}$$
where $b(f,g)=\sum\tr_{q^n/q}(a_i b_i)$ for $f(x)=\sum a_i x^{q^i}$ and $f(x)=\sum b_i x^{q^i}$.

It is not difficult to show that two additive codes $\Ca$ and $\Ca'$ are equivalent if and only if the codes $\Ca^\perp$ and $\Ca'^\perp$ are equivalent. Furthermore, Delsarte \cite{delsarte1978bilinear} proved the following result.

\begin{lemma}\label{item27}
	Let $\Ca$ be a $\Fq$-linear code. Then $\Ca$ is an MRD code if and only if its Delsarte dual  $\Ca^\perp$ is an MRD code.
\end{lemma}

An important relation between these invariants are the adjoint and Delsarte dual operation, given by
$$\mathcal{N}_r(\Ca^\perp)=\widehat{\mathcal{N}_r(\Ca)}=\mathcal{N}_m(\widehat{\Ca})$$
and 
$$\mathcal{N}_m(\Ca^\perp)=\widehat{\mathcal{N}_m(\Ca)}=\mathcal{N}_r(\widehat{\Ca}),$$
were stated in \cite[Proposition $4.2$]{lunardon2017kernels}.

\section{Equivalence}\label{sec14}

We say that a code is additive if the sum of two elements is also an element of the code. For $L_1,L_2\in\Lx$, it is not difficult to prove that $\hl$ is an additive code. Hence, the equivalence between two codes $\hl$ and $\hm$ is given by the existence of two bijective linearized polynomials $\phi_1,\phi_2\in\mathscr{L}_{n,q}[x]$ and $\rho\in\aut(\Fq)$ such that 
$$\hm=\{\phi_1\circ f^\rho\circ\phi_2:f\in\hl\}.$$

In this case, we say that $(\phi_1,\phi_2,\rho)$ is an equivalence map between $\hl$ and $\hm$. As $\Fq^*$ is a cyclic group, it is not difficult to show that $\aut(\Fq)=\langle \sigma \rangle$ where $\sigma(a):=a^p$ is \textit{Frobenius automorphism} from $\Fq$ to $\Fp$ and $p$ is the characteristic of $\Fq$. Throughout this paper, we always consider $k\leq n-1$. From now on, $L_1,L_2,M_1,M_2,L,M$ will denote linearized polynomials in $\Lx\backslash\{0\}$. Our main result in this section presents conditions on $L_1,L_2,M_1$ and $M_2$ for which the codes $\hl$ and $\hm$ are equivalent. In order to describe these conditions, the following lemmas will be needed.

\begin{lemma}\label{item28}
	Let $r,s,k$ and $n\geq 4$ be positive integers such that $\gcd(n,r)=\gcd(n,s)=1$, $2\leq k\leq n-2$ and $1\leq r,s<n$. Let $\Z_n$ denote the ring of integers module $n$ and let $\Gamma_{r,s,k}=\{tr-is\pmod{n}: \ 1 \leq i\leq k-1\text{ and }k+1\leq t\leq n-1\}$. If $k\geq \tfrac{n}{2}$, then the following hold:
	$$\Gamma_{r,s,k}=\begin{cases}
		\Z_n\backslash\{\overline{-s},\,\overline{0},\,\overline{s}\}, \text{ if }r=s;\\
		\Z_n\backslash\{\overline{-s(k-1)},\,\overline{-sk},\,\overline{-s(k+1)}\}, \text{ if }r=n-s;\\
		\Z_n\backslash\{\overline{-2s}\}, \text{ if }k+1=n-3\text{ and }r\equiv 2s\pmod{n};\\
		\Z_n\backslash\{\overline{6s}\}, \text{ if }k+1=n-3\text{ and }r\equiv -2s\pmod{n};\\
		\Z_n\backslash\{\overline{-s},\overline{-2s}\}, \text{ if }k+1=n-2\text{ and }r\equiv 2s\pmod{n};\\
		\Z_n\backslash\{\overline{5s},\overline{4s}\}, \text{ if }k+1=n-2\text{ and }r\equiv -2s\pmod{n};\\
		\Z_n\backslash\{\overline{-3s}\}, \text{ if }k+1=n-2\text{ and }r\equiv 3s\pmod{n};\\
		\Z_n\backslash\{\overline{6s}\}, \text{ if }k+1=n-2\text{ and }r\equiv -3s\pmod{n};\\
		\Z_n\backslash\{\overline{-r},\overline{-r+s},\overline{-r+2s}\}, \text{ if }k+1=n-1;\\
		\Z_n, \text{ otherwise}.\\
	\end{cases}$$
\end{lemma}

\begin{proof}
	The first cases can be obtained by straightforward computations, then we will only prove the last case. We also can assume $r< \tfrac{n}{2}$ replacing $r$ by $n-r$ if it is necessary. As $\gcd(s,n)=1$, we can suppose $s=1$ by multiplying every element by $s^{-1}$. We note that for each $t$ satisfying $k+1\leq t\leq n-1$, we have that
	$$[(t-1)r-1,tr-1]\subset\{tr-is\pmod{n}: \ 1 \leq i\leq k-1\text{ and }k+1\leq t\leq n-1\}$$
	since $k-1\geq r$. Then our result follows by observing that $r(n-1)-1\geq n+r(k+1)-(k-1)-1.$ $\hfill\qed$
\end{proof}

We note that a similar result can be obtained for the case $k\leq \tfrac{n}{2}$ by changing that role of $r$ and $s$.

In the following result, let $r$ and $s$ be integers satisfying $0\leq r<n$ and $0\leq s<n$.

\begin{theorem}\label{item4}
	Let $n\geq 4$ and let $k$ be an integer. Let $(\phi_1,\phi_2,\rho)$ be an equivalence map between $\hl$ and $\hm$. If $2< k< n-2$, then the following hold:
	\begin{enumerate}
		\item either $r=s$ or $r=n-s$;
		\item if $r=s$, then there exist elements $\alpha,\beta\in\Fqn$ and an integer $l$ satisfying $0\leq l\leq n-1$ such that $\phi_1(x)=\alpha x^{q^l}$ and $\phi_2(x)=\beta x^{q^{n-l}}$;
		\item if $r=n-s$, then there exist elements $\alpha,\beta\in\Fqn$ and an integer $l$ satisfying $0\leq l\leq n-1$ such that $\phi_1(x)=\alpha x^{q^{l-sk}}$ and $\phi_2(x)=\beta x^{q^{n-l}}$.
	\end{enumerate} 
	Furthermore, if there exists no $\gamma\in\Fqn^*$ such that $L_1(c)=\gamma L_2(c)$ for all $c\in\Fqn$ or $L_1(c)=\gamma L_2(c)^{q^s}$ for all $c\in\Fqn$, then \textit{(1)},  \textit{(2)} and \textit{(3)} hold for $k=2$ and $k=n-2$.
\end{theorem}

\begin{proof}
 Along the proof we will use that $k\geq \tfrac{n}{2}$ in order to use the previous lemma. If $k\leq \tfrac{n}{2}$, then we only need to change the role of $s$ and $r$ in order to use the same result. For any $y\in\Fqn$ and integer $i$ satisfying $1\leq i\leq k-1$, we have $\phi_1\circ y^\rho x^{q^{si}}\circ\phi_2\in\hm$. If $\phi_1(x)=\sum_{m=0}^{n-1} a_m x^{q^m}$ and $\phi_2(x)=\sum_{j=0}^{n-1} b_j x^{q^j}$, then
	$$\phi_1\circ y^\rho x^{q^{si}}\circ\phi_2=\phi_1\left(y^\rho\sum_{j=0}^{n-1} b_j^{q^{si}} x^{q^{j+si}}\right)$$
and therefore
	$$\phi_1\circ y^\rho x^{q^{si}}\circ\phi_2=\sum_{m=0}^{n-1} a_m y^{\rho q^m} \sum_{j=0}^{n-1} b_j^{q^{si+m}} x^{q^{j+si+m}}=\sum_{t=0}^{n-1} \left[\sum_{si+j+m=tr} a_m y^{\rho q^m} b_j^{q^{si+m}} \right]x^{q^{tr}}.$$
Hereafter, we consider the indexes of $a_i$ and $b_i$ module $n$, it means that $a_m:=a_i$ and $b_m:=b_i$ for every $m\equiv i\pmod{n}$. As $\phi_1\circ y^\rho x^{q^{si}}\circ\phi_2\in\hm$, we must have
\begin{equation}\label{item6}
T_{i,t}(y^\rho):=\sum_{si+j+m=tr} a_m y^{\rho q^m} b_j^{q^{si+m}}=0
\end{equation}
for all $k+1\leq t\leq n-1$, $y\in\Fq$ and $1\leq i\leq k-1$. We can rewrite $T_{i,t}(y^\rho)$ as
$$\sum_{m=0}^{n-1} a_m\,  b_{tr-m-si}^{q^{si+m}}\, y^{\rho q^m}.$$
We note that $T_{i,t}(x)\in\Lx$, then the Equation \eqref{item6} and Lemma \ref{item5} together imply that
\begin{equation}\label{item29}
a_m b_{tr-m-si}=0
\end{equation}
for all $0\leq m\leq n-1$, $k+1\leq t\leq n-1$ and $1\leq i\leq k-1$. We assume $r=s$. In this case, by Equation \eqref{item29} and Lemma \ref{item28} there exists a non-negative integer $l\leq n-1$ for which one of following holds:

\begin{enumerate}[(a)]
	\item $a_j= 0$ for every $j\neq l$ and $b_j= 0$ for every $j\notin\{n-s-l,n-l,n+s-l \};$
	\item $b_j= 0$ for every $j\neq l$ and $a_j= 0$ for every $j\notin\{n-s-l,n-l,n+s-l \};$
	\item $a_j= 0$ for every $j\notin\{l,l+s\}$ and $b_j= 0$ for every $j\notin\{n-s-l,n-l\};$
	\item $b_j= 0$ for every $j\notin\{l,l+s\}$ and $a_j= 0$ for every $j\notin\{n-s-l,n-l\}.$
\end{enumerate}

How $\phi_1$ and $\phi_2$ can be switched, we may assume without loss of generality that $(a)$ or $(c)$ holds. Now let us compute the image of $L_1(c)x+L_2(c)x^{q^{sk}}$ under the equivalence map $(\phi_1,\phi_2,\rho)$. In the case $(a)$, we have that $\phi_1\circ(L_1(c)x+L_2(c)x^{q^{sk}})^\rho\circ\phi_2$ equals
$$a_l\Big(L_1(c)^{\rho q^l}(b_{n-s-l}^{q^l}x^{q^{n-s}}+b_{n-l}^{q^l}x+b_{n+s-l}^{q^l}x^{q^{n+s}})+L_2(c)^{\rho q^l}(b_{n-s-l}^{q^l}x^{q^{sk-s}}+b_{n-l}^{q^l}x^{q^{sk}}+b_{n+s-l}^{q^l}x^{q^{sk+s}})\Big).$$
As $\phi_1\circ(L_1(c)x+L_2(c)x^{q^{sk}})^\rho\circ\phi_2\in\hm$ (and under the assumption that there exists no $\gamma\in\Fqn^*$ such that $L_1(c)=\gamma L_2(c)$ for all $c\in\Fqn$ in the case $k=2$ and $k=n-2$), it follows that $b_{n-s-l}=0$ and $b_{n+s-l}=0$. Then our result is shown in this case.

In the case $(c)$, we have that $\phi_1\circ(L_1(c)x+L_2(c)x^{q^{sk}})^\rho\circ\phi_2$ equals
$$\begin{aligned}& a_l\Big(L_1(c)^{\rho q^l}(b_{n-s-l}^{q^l}x^{q^{n-s}}+b_{n-l}^{q^l}x)+L_2(c)^{\rho q^l}(b_{n-s-l}^{q^l}x^{q^{sk-s}}+b_{n-l}^{q^l}x^{q^{sk}})\Big)\\
& +a_{l+s}\Big(L_1(c)^{\rho q^{l+s}}(b_{n-s-l}^{q^{l+s}}x+b_{n-l}^{q^{l+s}}x^{q^s})+L_2(c)^{\rho q^{l+s}}(b_{n-l}^{q^{l+s}}x^{q^{sk}}+b_{n-l}^{q^{l+s}}x^{q^{sk+s}})\Big).
\end{aligned}$$
As $\phi_1\circ(L_1(c)x+L_2(c)x^{q^{sk}})^\rho\circ\phi_2\in\hm$ (and under the assumption that there exists no $\gamma\in\Fqn^*$ such that $L_1(c)=\gamma L_2(c)^{q^s}$ for all $c\in\Fqn$ in the case $k=2$ and  $k=n-2$), it follows that $ a_l L_1(c)^{\rho q^l}b_{n-s-l}^{q^l}=0$ and $a_{l+s} L_2(c)^{\rho q^{l+s}}b_{n-l}^{q^{l+s}}=0$. Then $a_{l}=0$ and $b_{n-l}=0$ (or $a_{l+s}=0$ and $b_{n-s-l}=0$) and our result follows. The case $r=n-s$ can be obtained in the same way.

 Now assume that $k=n-2$ (the same for $k=2$) and $r\not\equiv\pm s\pmod{n}$. Under the assumption that there exists no $\gamma\in\Fqn^*$ such that $L_1(c)=\gamma L_2(c)$ for all $c\in\Fqn$, it follows (by the same proof done in the case $n=s$ using the Equation \eqref{item29} and Lemma \ref{item28}) that there exist elements $\alpha,\beta\in\Fqn^*$ and an integer $l$ satisfying $0\leq l\leq n-1$ such that $\phi_1(x)=\alpha x^{q^l}$ and $\phi_2(x)=\beta x^{q^{n-l+s-r}}$. We remind that each element $f(x)=\eta_0 x+\eta_k x^{q^{sk}}\in\hm$ is such that $\eta_0=M_1(c)$ and $\eta_k=M_2(c)$ for some $c\in\Fqn$.  As $r-s\not\equiv 0\pmod{n}$ and $r-s\not\equiv ks\pmod{n}$, we must have
 $$\phi_1\circ y^\rho x^{q^{s(r-s)}}\circ\phi_2=\alpha y^{\rho q^l}\beta^{q^{s(r-s)+l}}x\in\hm$$
for all $y\in\Fqn$, which is a contradiction since $M_2(x)\not\equiv 0$. Therefore $\hl$ is not equivalent to $\hm$ in this case.

 For the remaining possible values of $r$ and $s$ (namely $r\not\equiv \pm s\pmod{n}$, $k\neq 2$ and $k\neq n-2$), the Equation \eqref{item29} and Lemma \ref{item28} imply that if $a_i\neq 0$ for some $i$, then $b_j=0$ only for $j\in \Z_n\backslash \Gamma_{r,s,k}$. It is straight to compute that if $\phi_1\circ f^\rho\circ\phi_2\in\hm$ for all $f\in\hl$, then $b_j=0$ for every $j$, which is a contradiction.  $\hfill\qed$
\end{proof}

The following lemma is a trivial result from linear algebra and it will be used in the next theorem.
\begin{lemma}\label{item17}
	Let $M(x),L(x)\in\Lx$. If $\Ima L=\Ima T$, then there exists a bijective linearized polynomial $T(x)\in\Lx$ such that $M(x)=L(T(x))$.
\end{lemma}

Throughout the rest of the section, we assume that there exists no $\gamma\in\Fqn^*$ such that $L_1(c)=\gamma L_2(c)$ for all $c\in\Fqn$ or $L_1(c)=\gamma L_2(c)^{q^s}$ for all $c\in\Fqn$
\begin{theorem}\label{item3}
	
	Let $n\geq 4$ and let $k$ be an integer satisfying $2\leq k\leq n-2$. The codes $\hl$ and $\hms$ are equivalent if and only if there exist $c,d\in\Fq$, $\rho=p^\nu\in\aut(\Fqn)$, an integer $l$ satisfying $0\leq l\leq n-1$ and a bijective linearized polynomial $T\in\Lxp$ such that $M_1(x)=ab L_1(T(x))^{p^\nu q^l}$ and $M_2(x)=ab^{ q^{sk}}L_2(T(x))^{p^\nu q^l}.$

\end{theorem}

\begin{proof}
	The necessity follows immediately from Theorem \ref{item4}. Let us prove the converse. Let $l\leq n-1$ be an integer, $\phi_1(x)=a x^{q^l}$,  $\phi_2= b^{q^{n-l}} x^{q^{n-l}}$ and $\rho\in\aut(\Fqn)$. Since $\hl=\langle L_1(a_0)x+L_2(a_0)x^{q^{sk}}, a_1x^{q^s},\ldots,a_{k-1}x^{q^{s(k-1)}} \rangle_{a_i\in\Fqn}$, we only need to prove that the set of images of these elements under the equivalence map is a basis of $\hms$. For $1\leq i\leq k-1$, we have that
	$$\phi_1\circ a_i^\rho x^{q^{si}}\circ \phi_2=a a_i^{\rho q^l}b^{q^{si}} x^{q^{si}}\in\hms$$
	for all $a_i\in\Fq$. For $a_0\in\Fq$,
	$$\phi_1\circ (L_1(a_0)x+L_2(a_0)x^{q^{sk}})^\rho\circ \phi_2=ab L_1(a_0)^{\rho q^l} x+ab^{ q^{sk}}L_1(a_0)^{\rho q^l} x^{q^{sk}}\in\hms$$
	where $\rho=p^m$ for an integer $m\geq0$. Since we need
	$$\langle ab L_1(a_0)^{\rho q^l} x+ab^{ q^{sk}}L_1(a_0)^{\rho q^l} x^{q^{sk}}, a a_i^{\rho q^l}b^{q^{si}} x^{q^{si}} \rangle_{1\leq i\leq k-1, a_i\in\Fqn}=\hms,$$
	 our result follows from Lemma \ref{item17}.$\hfill\qed$
	
\end{proof}

A very similar result can be obtained for the case $r=n-s$, where the role of $L_1$ and $L_2$ is changed. As a consequence of Theorem \ref{item3}, one can easily show Theorems $11,12$ and $13$ of Trombetti-Zhou \cite{trombetti2018new}.

\begin{corollary}\label{item19}
	For $k$ and $n$ satisfying the hypothesis of the Theorem \ref{item3}, $\mathcal{H}_{k,s}(x,L(x))$ is equivalent to $\mathcal{H}_{k,s}(x,M(x))$ if and only if there exist $a,b\in\Fqn$, $\rho=p^\nu\in\aut(\Fqn)$ and an integer $l$ satisfying $0\leq l\leq n-1$ such that $M(x)=ab^{q^{sk+l}}L((x/ab^{q^l})^{p^{\lambda n-\nu}q^{n-l}})^{\rho q^l}$.
\end{corollary}
	This result generalize the Theorem $4.4$ from Lunardon, Trombetti and Zhou \cite{lunardon2018generalized} for the case $L(x)=\eta x^{q^{sg}}$ and $M(x)=\theta x^{q^{sh}}$.

\section{Codes of the form $\mathcal{H}_{k,s}(x,L(x))$}\label{sec15}

Throughout this section,  we let $L=\sum \eta_i x^{q^{si}}$ denotes an $q^s$-linearized polynomial in $\Lx$. In this section we completely characterize the invariants of the codes $\mathcal{H}_{k,s}(x,L(x))$. 

\begin{lemma}\label{item16}
	For $a,b\in\Fqn$, we have
	 $$\tr_{q^n/q}\big(b L(a)\big)=\tr_{q^n/q}\big(a\hat{L}(b)\big),$$
	 where $\hat{L}$ is the adjoint of $L$.
\end{lemma}

\begin{proof}
	We have
	$$\tr_{q^n/q}\big(b L(a)\big)=\sum_{i=0}^{n-1}\left[bL(a)\right]^{q^{si}}=\sum_{i=0}^{n-1}\left[b\sum_{j=0}^{n-1}\eta_j a^{q^{sj}}\right]^{q^{si}}=\sum_{i=0}^{n-1}\sum_{j=0}^{n-1}b^{q^{si}}\eta_j^{q^{si}}a^{q^{s(i+j)}}$$
	and letting $i+j=k$ it follows that
	$$\sum_{i=0}^{n-1}\sum_{j=0}^{n-1}b^{q^{si}}\eta_j^{q^{si}}a^{q^{s(i+j)}}=\sum_{i=0}^{n-1}\sum_{l=0}^{n-1}a^{q^{sl}}\eta_{l-i}^{q^{si}}b^{q^{si}}=\sum_{l=0}^{n-1}\left[a\sum_{i=0}^{n-1}\eta_{l-i}^{q^{s(i-l)}}b^{q^{s(i-l)}}\right]^{q^{sl}}=\tr_{q^n/q}\big(a \hat{L}(b)\big),$$
	proving the lemma.$\hfill\qed$
\end{proof}

\begin{theorem}\label{item7}The Delsarte dual of $\hxl$ is equivalent to $\mathcal{H}_{n-k,s}(x,-\hat{L}(x))$, where $\hat{L}$ is the adjoint of $L$.
\end{theorem}

\begin{proof}
	We note that $\mathcal{H}_{n-k,s}(x,-\hat{L}(x))$ is equivalent to
	$$\mathcal{J}:=\left\{-\hat{L}(b_k)+\sum_{i=k}^{n-1}b_i x^{q^{si}}:b_i\in\Fqn\right\}.$$
	Since the dimension of the code $\mathcal{J}$ over $\Fq$ is $n-k$, we only need to show that $\mathcal{J}\subset\hxl^\perp$. If $g(x)=\sum_{i=0}^{n-1} b_ix^{q^{si}}\in\mathcal{J}$ and $f(x)=\sum_{i=0}^{n-1} a_0x^{q^{si}}\in\hxl$, then
	$$b(f,g)=\sum_{i=0}^{n-1} \tr_{q^n/q}(a_ib_i)=\tr_{q^n/q}\big(b_k L(a_0)\big)-\tr_{q^n/q}\big(a_0\hat{L}(b_k)\big)$$
	and our result follows from Lemma \ref{item16}.$\hfill\qed$
\end{proof}

This result was already established earlier by Lunardon, Trombetti and Zhou \cite[Proposition $4.2$]{lunardon2018generalized} for the case where $L(x)=\eta x^{q^h}$.

\begin{lemma}\label{item14}
	Assume $L(x)=\sum_{i=0}^{M}a_i x^{q^{e_i}}$, where $a_i\in\Fqn^*$ and $M$ is a positive integer and $e_0<e_1<\ldots<e_L$ are non-negative integers. Let $d=\gcd(e_1,\ldots e_L,n)$. Then $L(\alpha x)=\alpha L(x)$ for all $x\in\Fqn$ if and only if $\alpha\in\F_{q^d}$.
\end{lemma}

\begin{proof}
	Suppose that $L(\alpha x)=\alpha L(x)$ for all $x\in\Fqn$. We have that
	$$\alpha\sum_{i=0}^{M}a_i x^{q^{e_i}}-\sum_{i=0}^{M}a_i (\alpha x)^{q^{e_i}}=\sum_{i=0}^{M}a_i[\alpha-\alpha^{q^{e_i}}]x^{q^{e_i}}=0$$
	for all $x$ and then $\alpha^{q^{e_i}}=\alpha$ from Lemma \ref{item5}, which implies that $\alpha\in\F_{q^d}$. The converse is trivial.$\hfill\qed$
	
\end{proof}

\begin{theorem}\label{item11}
	Let $n\geq 3$. Assume $L(x)=\sum_{i=0}^{M}\eta_{e^i} x^{q^{s e_i}}$, where $\eta_i\in\Fqn^*$ and $M$ is a positive integer and $e_0<e_1<\ldots<e_M<n$ are non-negative integers. Let $d=\gcd(e_1,\ldots e_L,n)$. The right nucleus of $\hxl$ is
	$$\mathcal{N}_r(\hxl)=\{ax:a\in\F_{q^d}\}.$$
	The middle nucleus of $\hxl$ is
	$$\mathcal{N}_m(\hxl)=\{ax:a\in\F_{q^d}\}.$$
\end{theorem}

\begin{proof} We will compute only the right nucleus, since the middle nucleus can be computed in a similar way by doing the needed changes. By duality presented in the Theorem \ref{item7} and the Delsarte dual operation we can suppose without loss of generality that $k\leq \frac{n}{2}$. We can write $L(x)$ as 
	$\sum_{j=0}^{n-1} \eta_j x^{q^{sj}}$
	by setting $b_j=0$ if $j\not\in\{e_i:i=0,\ldots,M\}$. Now let $g(x)=\sum_{i=0}^{n-1} b_ix^{q^{si}}\in\mathcal{N}_r(\hxl)$. 
	
	\textbf{  Claim $1$.} $b_i=b_{i-k}=0$ for all $i=k+1,\ldots,n-1$.
	
	\textit{  Proof of the Claim $1$.} For every $c\in\Fqn$, we have that 
	$$g\big(cx+L(c)x^{q^{sk}}\big)=\sum_{i=0}^{n-1}x^{q^{si}}\left(b_i c^{q^{si}}+b_{i-k}L(c)^{q^{s(i-k)}}\right)\in\hxl.$$
	By Lemma \ref{item5}, for $i$ satisfying $k+1\leq i\leq n-1$ we have that the linearized polynomial $b_i c^{q^{si}}+b_{i-k}L(c)^{q^{s(i-k)}}$ is identically null. In particular, $b_{i-k} \eta_{j}^{q^{s(i-k)}}=0$ for all $j\neq k$. Since $M\geq 1$, there exists an integer $\delta\neq k$ such that $\eta_\delta \neq0$, then $b_{i-k} \eta_\delta^{q^{s(i-k)}}=0$ implies that $b_{i-k}=0$. Besides that, since $b_i+b_{i-k}\eta_k^{q^{s(i-k)}}=0$ we have $b_{i}=0$, proving the claim.
	
	\textbf{  Claim $2$.} If $2\leq k\leq n-2$, then $b_i=0$ for all $i=2,\ldots,n-2$.
	
	\textit{  Proof of the Claim $2$.} For $j$ satisfying $1\leq j\leq k-1$, we have that $g(x^{q^{sj}})=\sum b_{i-j} x^{q^{si}}$. Therefore $b_{i-j}=0$ for all $i=k+1,\ldots, n-1$, proving the claim.
	
	The claims $1$ and $2$ together imply that $b_i=0$ for all $i\neq0$. Now we consider $g(x)=b_0 x\in\mathcal{N}_r(\hxl)$. Since 
	$$b_0 c x+b_0 L(c)x^{q^{sk}}\in\hxl,$$
	we must have $L(b_0 c)=b_0 L(c)$ for all $c\in\Fqn$ and our result follows from Lemma \ref{item14}.$\hfill\qed$
	
\end{proof}

\section{The Automorphism Group}\label{sec16}

The \textit{automorphism group} of an additive code $\Ca$ is given by
$$\aut(\Ca)=\{(\phi_1,\phi_2,\rho):\phi_1\circ \Ca^\rho\circ\phi_2=\Ca\}.$$
Sheekey computed the automorphism group of Gabidulin codes (Theorem $4$ of \cite{sheekey2015new}) and Twisted Gabidulin codes (Theorem $7$ of \cite{sheekey2015new}).  With Theorem \ref{item3} we are able to describe the automorphism group of $\hxl$. In this section, we will be always assuming that $k$ is an integer satisfying $2\leq k\leq n-2$ and that there exists no $\gamma\in\Fqn^*$ such that $L_1(c)=\gamma L_2(c)$ for all $c\in\Fqn$ or $L_1(c)=\gamma L_2(c)^{q^s}$ for all $c\in\Fqn$. For a set $\mathcal{I}\subseteq[0,n-1]$, let 
$D(\mathcal{I})=\{i-j:i,j\in\mathcal{I}, i>j\}$
be the set of distinct differences in $\mathcal{I}$. Let $i_1,\ldots, i_u$ represent the elements of $D(\mathcal{I})$. In order to present the automorphism group of $\hxl$, we define the function
$$\kappa_{q^n}(\mathcal{I})=\gcd(q^{i_1}-1,\ldots, q^{i_u}-1,q^n-1)=q^{\gcd(i_1,\ldots,i_u,n)}-1.$$
Along the proof of the following result we will extensively use the well-known fact that $\gcd(q^i-1,q^j-1)=q^{\gcd(i,j)}-1$ for positive integers $i,j$ and $q$. From now on, for $d$ a divisor of $q^n-1$, let $\chi_d$ be the multiplicative character of $\Fqn^*$ of order $j$. For convenience we extend a multiplicative character setting $\chi_d(0)=0$. 

\begin{theorem}\label{item18}
	Let $\mathcal{I}\subseteq[0,n-1]$ be a nonempty set and $d=\gcd(q^k-1,q^n-1)$. Assume $L(x)=\sum_{i=0}^{n-1} \eta_i x^{q^i}$ where $\eta_i\neq0$ if $i\in\mathcal{I}$ and $\eta_i=0$ otherwise. The automorphism group $\aut\big(\hxl\big)$ is given by
	$$\left\{\!\big(a x^{q^l},b x^{q^{n-l}},p^\nu\big)\!\!:a,b\in\Fqn^*,0\leq l< n,0\leq\nu< \lambda n\text{ and }\eta_i^{p^\nu q^l-1}=\tfrac{(ab^{q^l})^{q^i}}{ab^{q^{l+sk}}}\text{ for all }i\in\mathcal{I}\right\}.$$
	Furthermore, if $|\mathcal{I}|\geq 2$ then 
	$$\left|\aut\big(\hxl\big)\right|\!= \kappa d \left|\left\{(\nu,l)\in \mathcal{B}\middle|
	\begin{aligned}
	&\exists\ \alpha,\beta\in\Fqn^*\text{ such that }\chi_d(\alpha\beta)=1\\
	&\text{          and  }\eta_i^{p^\nu q^l-1}=\alpha \beta^{q^i} \text{ for all } i\in\mathcal{I}\\
	\end{aligned}\right\}\right| $$
	where $\mathcal{B}=[0,\lambda n-1]\times[0,n-1]$ and $\kappa=\gcd\left(\kappa_{q^n}(\mathcal{I}),\tfrac{q^n-1}{d}\right)$.
	
\end{theorem}

\begin{proof}
	By Theorem \ref{item4}, we can assume that an element of $\aut\big(\hxl\big)$ is of the form $\big(a x^{q^l},b x^{q^{n-l}},p^\nu\big)$ where $a,b\in\Fqn^*$, $\rho\in\aut(\Fqn)$ and $l$ is an integer satisfying $0\leq l\leq n-1$. By Theorem \ref{item3}, we have that
	$$\sum_{i=0}^{n-1} \eta_i x^{q^i}=L(x)=ab^{q^{sk+l}}L((x/ab^{q^l})^{p^{n\lambda -\nu}q^{n-l}})^{p^\nu q^l}=ab^{q^{sk+l}}\sum_{i=0}^{n-1} \eta_i^{p^\nu q^l} \frac{1}{(ab^{q^l})^{q^i}}\, x^{q^i}.$$ 
	The first part of our result follows from Lemma \ref{item5}. Now let
	$$\Delta_{\nu,l}=\{(\alpha,\beta)\in\Fqn^2:\eta_i^{p^\nu q^l-1}=\alpha \beta^{q^i}\, \text{ for all } i\in\mathcal{I}\}.$$
	
	\textbf{  Claim $1$.} For each pair $(\alpha,\beta)\in \Delta_{\nu,l}$ the following hold:
	\begin{enumerate}[(a)]
		\item If $\chi_d(\alpha\beta)=1$, then there exist exactly $d$ pairs $(a,b)\in\Fqn^2$ such that $\big(a x^{q^l},b x^{q^{n-l}},p^\nu\big)\in\aut\big(\hxl\big)$;
		\item If $\chi_d(\alpha\beta)\neq1$, then there exists no pair $(a,b)\in\Fqn^2$ such that $\big(a x^{q^l},b x^{q^{n-l}},p^\nu\big)\in\aut\big(\hxl\big)$.
	\end{enumerate}
	
	\textit{  Proof of the Claim $1$.} 
	Let $(\nu,l)\in \mathcal{B}$ and assume that there exist $\alpha,\beta\in\Fqn^*$ such that $\eta_i^{p^\nu q^l-1}=\alpha \beta^{q^i}$. If $\big(a x^{q^l},b x^{q^{n-l}},p^\nu\big)\in\aut\big(\hxl\big)$,  then a straightforward computation shows that
	$$b^{q^{sk}-1}=(\alpha\beta)^{-q^{n-l}}\text{ and }a=\beta b^{-q^l}.$$
	Therefore $\chi_d(\alpha\beta)=1$. Furthermore, for each $\beta$ there exist $d$ elements $b\in\Fqn$ such that $b^{q^{sk}-1}=(\alpha\beta)^{-q^{n-l}}$. Then for each pair $(\alpha,\beta)\in \Delta_{\nu,l}$, there exist $d$ pairs $(a,b)\in\Fqn^2$ such that $\big(a x^{q^l},b x^{q^{n-l}},p^\nu\big)\in\aut\big(\hxl\big)$, proving the claim.
	
	Now we let
	$$\Delta_{\nu,l}'=\{(\alpha,\beta)\in\Fqn^2:\eta_i^{p^\nu q^l-1}=\alpha \beta^{q^i}\, \text{ for all } i\in\mathcal{I}\text{ and }\chi_d(\alpha\beta)=1\}.$$
	\textbf{  Claim $2$.} For each pair $(\nu,l)\in \mathcal{B}$ we have that either $|\Delta_{\nu,l}'|=0$ or $|\Delta_{\nu,l}'|=\gcd\left(\kappa_{q^n}(\mathcal{I}),\tfrac{q^n-1}{d}\right)$.
	
	\textit{  Proof of the Claim $2$.} Let $(\alpha_1,\beta_1)$ and $(\alpha_2,\beta_2)$ be two elements (distinct or not) in $\Delta_{\nu,l}'$. For $i>j$ elements in $\mathcal{I}$, we have that
	$$\beta_1^{q^i-q^j}=\eta_i^{p^\nu q^l-1}\eta_j^{-(p^\nu q^l-1)}=\beta_2^{q^i-q^j}$$
	and then it follows that $\beta_1=\delta\beta_2$ where $\delta$ is a $(q^{\gcd(i-j,n)}-1)$-th root of unity in $\Fqn$. As $i,j$ was taken arbitrarily, it follows that $\beta_2=\xi\beta_1$ where $\xi$ is a $\kappa_{q^n}(\mathcal{I})$-th root of unity in $\Fqn$ and then
	$$\Delta_{\nu,l}'\subseteq\{\xi^i\beta_1: 0\leq i\leq\kappa_{q^n}(\mathcal{I})\}$$
	for $\xi$ a primitive $\kappa_{q^n}(\mathcal{I})$-th root of unity. If $\xi^i\beta_1\in\Delta_{\nu,l}'$ then
	$$1=\chi_d(\xi^i \beta_1)=\chi_d(\xi^i)=\xi^{i\frac{q^n-1}{d}},$$
	which implies that  $i$ is a multiple of $\kappa_{q^n}(\mathcal{I})/\gcd\left(\kappa_{q^n}(\mathcal{I}),\tfrac{q^n-1}{d}\right).$ Therefore 
	$$|\Delta_{\nu,l}'|=\frac{\kappa_{q^n}(\mathcal{I})}{\tfrac{\kappa_{q^n}(\mathcal{I})}{\gcd\left(\kappa_{q^n}(\mathcal{I}),\tfrac{q^n-1}{d}\right)}}=\gcd\left(\kappa_{q^n}(\mathcal{I}),\tfrac{q^n-1}{d}\right),$$
	 proving the claim.
	
	The Claims $1$ and $2$ together show our result.$\hfill\qed$
\end{proof}

As an immediate consequence we have the following result. 

\begin{corollary}
	Let $L$ be a linearized polynomial and $n,k,d$ and let $\kappa$ be integers under the same conditions as Theorem \ref{item18}. Let $\mathcal{I}\subseteq[0,n-1]$ be a set with $|\mathcal{I}|\geq 2$. Then
	$$\left|\aut\big(\hxl\big)\right|=\kappa d\frac{\lambda n^2}{\tau(L)}\leq\kappa d \lambda n^2$$
	where
	$$\tau(L)=\min\left\{m|\lambda n\middle|\exists \alpha,\beta\in\Fqn^*\text{ such that }\chi_d(\alpha\beta)=1\text{ and }\eta_i^{p^m-1}=\alpha\beta^{q^i}\text{ for all }i\in\mathcal{I}\right\}.$$
\end{corollary}

\begin{proof}
Let
$$A=\{m|\lambda n :\exists \alpha,\beta\in\Fqn^*\text{ such that }\chi_d(\alpha\beta)=1\text{ and }\eta_i^{p^m-1}=\alpha\beta^{q^i}\text{ for all }i\in\mathcal{I}\}.$$
We only need to show that if $m_1,m_2\in A$, then $\gcd(m_1,m_2)\in A$, since for each $m\in A$, it is easy to note that $lm\in A$ for any integer $l$. For $m_1,m_2$ and $m=\gcd(m_1,m_2)$, let  $a$ and $b$ be integers such that $a(p^{m_1}-1)+b(p^{m_2}-1)=p^m-1$. Suppose that $\alpha_1,\alpha_2,\beta_1,\beta_2\in\Fqn^*$ are such that $\eta_i^{p^{m_1}-1}=\alpha_1\beta_1^{q^i}$ and $\eta_i^{p^{m_2}-1}=\alpha_2\beta_2^{q^i}\text{ for all }i\in\mathcal{I}$. Then
$$\eta_i^{p^{m}-1}=\eta_i^{a(p^{m_1}-1)+b(p^{m_2}-1)}=\left(\alpha_1\beta_1^{q^i}\right)^a\left(\alpha_2\beta_2^{q^i}\right)^b=(\alpha_1^a\alpha_2^b)(\beta_1^a\beta_2^b)^{q^i}\text{ for all }i\in\mathcal{I}.$$
Our result follows fro Theorem \ref{item18} by observing that $\chi_d(\alpha_1\beta_1)=1$ and $\chi_d(\alpha_2\beta_2)=1$ imply $\chi_d(\alpha_1^a\alpha_2^b\beta_1^a\beta_2^b)=1$.$\hfill\qed$
\end{proof}

In particular, in the case where there exist $\alpha,\beta\in\Fqn^*$ such that $\chi_d(ab)=1$ and $\eta_i=\alpha \beta^i$ for all $i\in\mathcal{I}$, the integers $\nu$ and $l$ can be chosen arbitrary and $a,b$ are fixed. Then
$$\left|\aut\big(\hxl\big)\right|=\kappa\lambda d n^2.$$

The remaining case is $L(x)=\eta x^{q^h}$. For this case, let $g$ be a primitive element of $\Fqn^*$ and let $\eta=g^u$, where $u$ is an integer.

\begin{corollary}\label{item23}
	If $\eta=g^u\in\Fqn$ and $d=\gcd(q^n-1,q^h-1,q^{sk-h}-1)$, then
	$$\left|\aut\big(\mathcal{H}_{k,s}(x,\eta x^{q^h})\big)\right|=d(q^n-1)\frac{\lambda n^2}{\tau(g^u,h)}$$
	where
	$$\tau(g^u,h)=\min\left\{m|\lambda n:(q^{\gcd(n,h,sk-h)}-1)\text{ divides }(u(p^m-1))\right\}.$$
\end{corollary}

\begin{proof}
	Let $\Delta=\left\{m|\lambda n:(q^{\gcd(n,h,sk-h)}-1)\text{ divides }(u(p^m-1))\right\}$. By Theorem \ref{item18}, if $\big(a x^{q^l},b x^{q^{n-l}},p^\nu\big)\in\aut\big(\mathcal{H}_{k,s}(x,\eta x^{q^h})\big)$ then
	$$(g^u)^{p^\nu q^l-1}=\tfrac{(ab^{q^l})^{q^h}}{ab^{q^{l+sk}}}=a^{q^h-1}b^{q^l(q^h-q^{sk})}=a^{q^h-1}b^{q^{l+h}(1-q^{sk-h})}.$$
	Set $d_1=q^h-1$, $d_2=1-q^{sk-h}$ and $d=\gcd(q^n-1,d_1,d_2)$. Te have that
	$$\left(g^{u(p^\nu q^l-1)}\right)^{\frac{q^n-1}{d}}=\left(a^{d_1}b^{q^{l+h}d_2}\right)^{\frac{q^n-1}{d}}=1$$
	and then $d|(u(p^\nu q^l-1))$. In particular, $(\nu+\lambda l)\in\Delta$. Hence, if $a=g^i$ and $b^{q^{l+h}}=g^j$, then
	\begin{equation}\label{item25}
	\left(g^{\frac{id_1}{d}}g^{\frac{jd_2}{d}}\right)^{d}=g^{id_1}g^{jd_2}=g^{u(p^\nu q^l-1)}=\left(g^{\frac{u(p^\nu q^l-1)}{d}}\right)^{d},
	\end{equation}
	whose number of pairs $(i,j)$ of solutions is equal to $d$ times the number of solutions of $g^{(id_1)/d}g^{(jd_2)/d}=g^{u(p^\nu q^l-1))/d}.$
	Then we only need to compute the number pairs $(i,j)$ satisfying the equation
	\begin{equation}\label{item22}
	i\left(\tfrac{d_1}{d}\right)+j\left(\tfrac{d_2}{d}\right)\equiv \tfrac{u(p^\nu q^l-1)}{d}\pmod{q^n-1},
	\end{equation}
	which is a Diophantine equation over $\Z_{(q^n-1)}$. For a pair $(i,j)$, let $f(i,j)=i\left(\tfrac{d_1}{d}\right)+j\left(\tfrac{d_2}{d}\right)$. Let $(i_e,j_e)\in \Z_{(q^n-1)}$ satisfying $f(i_e,j_e)=e$ with $e:=\gcd(d_1/d,d_2/d)$. We recall that $e$ is an unity in $\Z_{(q^n-1)}$, then $e^{-1}\in\Z_{(q^n-1)}$. For $A\in\Z_{(q^n-1)}$, let
	\begin{equation}\label{item24}
		\Upsilon_A=\{(i,j)\in\Z_{(q^n-1)}^2: f(i,j)\equiv A\pmod{q^n-1}\}
	\end{equation} 
	 be the set of solutions of $f(i,j)=A$ and let $n_A=|\Upsilon_A|$. It is direct to verify that $\sum_A n_A\le (q-1)^2$ and $(i_A,j_A):=(i_0e^{-1} A,j_0 e^{-1} A)\in\Upsilon_A$. Furthermore, if $(i_{A,t},j_{A,t})$ is given by $i_{A,t}=i_A+t\frac{d_2}{d}$ and $j_{A,t}=j_A-t\frac{d_1}{d}$ with $t\in\Z_{(q^n-1)}$, then $ (i_{A,t},j_{A,t})\in\Upsilon_A$. Set $d_1'=\gcd(q^n-1,d_1/d)$ and $d_2'=\gcd(q^n-1,d_2/d)$. We observe that that $i_{A,t}=i_{A,t'}$ if and only if $t\equiv t'\pmod{\tfrac{q^n-1}{d_1'}}$, then $(i_{A,t},j_{A,t})=(i_{A,t'},j_{A,t'})$ if and only if $t\equiv t'\pmod{\lcm(\tfrac{q^n-1}{d_1'}, \tfrac{q^n-1}{d_2'})}$. Since $\lcm(\tfrac{q^n-1}{d_1/d},\tfrac{q^n-1}{d_2/d})=q^n-1$, it follows that $n_A\ge q-1$. Since $\sum_A n_A\le (q-1)^2$ and $n_A\ge q-1$, we must have $n_A=q-1$ for all $A\in\Z_{(q^n-1)}$. In particular, $\Upsilon_{\ell}=q-1$ for $\ell=\tfrac{u(p^\nu q^l-1)}{d}$ and then the number of solutions of Eq.~\eqref{item25} is $d(q^n-1)$. More generally, we have that the number of pairs $(a,b)$ for which $\big(a x^{q^l},b x^{q^{n-l}},p^\nu\big)\in\aut\big(\mathcal{H}_{k,s}(x,\eta x^{q^h})\big)$ is $d(q^n-1)$ provided $(\nu+\lambda l)\in\Delta$ and then we only need to compute the number of such pairs $(\nu,l)$.
	 
	  It is direct to verify the number of pairs $(\nu,l)$ such that $(\nu+\lambda l)\in\Delta$ is exactly $n\cdot|\Delta| $. Similarly to the proof of the previous result, we can show that 
	 $$\Delta=\{l\tau(g^u,h):0\leq l<q^n-1\}$$ and then $|\Delta|=\tfrac{\lambda n}{\tau(g^u,h)}$, from where our result follows.$\hfill\qed$
\end{proof}
In particular, Corollary \ref{item23} gives us the number of automorphisms of Generalised Twisted Gabidulin codes.
\begin{example}
	Let $g$ be a primitive element of $\Fqn$ and let $h$ be an positive integer. For $u$ an integer such that $\gcd(u,n)=1$, the number of elements of the automorphism group of the code $\mathcal{H}_{k,s}(x,g^u x^{q^h})$ equals $n^2(q^n-1)$, since we have that $\tau(g^u,h)=\lambda\gcd(n,h,sk-h)$. 
	
\end{example}

\section{Acknowledgments}

This study was financed in part by the Coordenação de Aperfeiçoamento de Pessoal de Nível Superior - Brasil (CAPES) - Finance Code 001. 

\bibliographystyle{ieeetr}
\bibliography{biblio}
	
\end{document}